\documentclass[11pt]{article}
\usepackage{latexsym,amssymb,amsmath,amsthm,enumerate,float,geometry,cite,tikz,bbm,algorithm2e}
\usepackage{xcolor}
\newtheorem{theorem}{Theorem}

\newtheorem{corollary}[theorem]{Corollary}

\newtheorem{claim}{Claim}

\usepackage{lineno}
\usepackage{setspace}
\allowdisplaybreaks

\begin{document}
\onehalfspace

\title{Unbalanced spanning subgraphs in edge labeled complete graphs\thanks{Research supported by research grant DIGRAPHS ANR-19-CE48-0013.}}  

\author{
St\'{e}phane Bessy$^1$ 
\and Johannes Pardey$^2$ 
\and Lucas Picasarri-Arrieta$^1$ 
\and Dieter Rautenbach$^2$} 

\date{}

\maketitle
\vspace{-10mm}
\begin{center}
{\small 
$^1$ LIRMM, Univ Montpellier, CNRS, Montpellier, France\\
\texttt{$\{$stephane.bessy,lucas.picasarri-arrieta$\}$@lirmm.fr}\\[3mm] 
$^2$ Institute of Optimization and Operations Research, Ulm University, Ulm, Germany\\
\texttt{$\{$johannes.pardey,dieter.rautenbach$\}$@uni-ulm.de}}
\end{center}

\begin{abstract}
Let $K$ be a complete graph of order $n$. 
For $d\in (0,1)$, let $c$ be a $\pm 1$-edge labeling of $K$ 
such that there are $d{n\choose 2}$ edges with label $+1$, 
and let $G$ be a spanning subgraph of
$K$ of maximum degree at most $\Delta$.  We prove the existence of an
isomorphic copy $G'$ of $G$ in $K$ such that the number of edges with label $+1$
in $G'$ is at least
$\left(c_{d,\Delta}-O\left(\frac{1}{n}\right)\right)m(G)$, 
where $c_{d,\Delta}=d+\Omega\left(\frac{1}{\Delta}\right)$ for fixed $d$, that is, 
this number visibly
deviates from its expected value when considering a uniformly random
copy of $G$ in $K$.  For $d=\frac{1}{2}$, and $\Delta\leq 2$, we
present more detailed results.\\[3mm] 
\end{abstract}

\section{Introduction}

Let $K$ be a complete graph with vertex set $[n]=\{ 1,\ldots,n\}$, 
and let $c:E(K)\to\{ \pm 1\}$ be a $\pm 1$-edge labeling of $K$.
The edge-labeling $c$ of $K$ is {\it balanced} 
if there are equally many {\it plus-edges} and {\it minus-edges},
that is, edges with label $+1$ and $-1$, respectively.
For a spanning subgraph $G$ of $K$, let $c(G)=c(E(G))=\sum\limits_{e\in E(G)}c(e)$, 
and let $m^+(G)$ and $m^-(G)$ be the number of plus-edges and minus-edges of $G$.
Note that $c(G)=m^+(G)-m^-(G)$. 
For a permutation $\pi$ from $S_n$,
let $G_\pi$ be the isomorphic copy of $G$ in $K$ with edge set $\{ \pi(u)\pi(v):uv\in E(G)\}$.

In the present paper, we study the structure of the set
$$\sigma_{(K,c)}(G)=\{ m^+(G_\pi):\pi\in S_n\}.$$
Our research is inspired by recent beautiful work of Caro, Hansberg, and Montejano \cite{cahamo}
on so-called {\it omnitonal} graphs.
Roughly speaking, 
a graph $G$ is said to be omnitonal 
if for every pair $(K,c)$, 
where the order $n$ of $K$ is sufficiently large and there are sufficiently many plus-edges and minus-edges in $K$,
and for every two non-negative integers $m^+$ and $m^-$ with $m(G)=m^++m^-$,
there is an isomorphic copy $G'$ of $G$ in $K$ with $m^+(G')=m^+$ and $m^-(G')=m^-$.
The key difference to the problems we study here 
is that the order of $K$ is necessarily much bigger than the order of $G$,
that is, 
the graph $G$ is far from being a spanning subgraph of $K$.
Quite surprisingly, exploiting recent strong results from Ramsey theory \cite{cumo,fosu},
Caro et al.~\cite{cahamo} achieve a very concise characterization of omnitonal graphs.
While being inspired by their work, requiring that $G$ is a spanning subgraph of $K$ 
drastically changes the nature of the problem.
The higher the density of a spanning graph $G$ is, the more every isomorphic copy of $G$ in $K$
is forced to reproduce the density of plus- and minus-edges in $(K,c)$.
Therefore, as a natural hypothesis excluding dense spanning subgraphs, 
we consider graphs of bounded maximum degree.

Another perspective on our results is that they correspond to relaxed versions of classical extremal problems,
which ask how many edges suffice to ensure the existence of a specific subgraph. 
In order to force a Hamiltonian cycle in a graph $G$ of order $n$, for instance, 
one needs to require at least ${n-1 \choose 2}+2$ edges in $G$, that is, 
the graph has to be almost complete with a density $m(G)/{n\choose 2}$ tending to $1$.
Our Theorem \ref{theorem1}(i) below can be rephrased to say 
that while a density of $1/2$ does not force the existence of a Hamiltonian cycle,
it forces the existence of $58\%$ of it, 
more precisely, of a Hamiltonian cycle in the complete graph on the same vertex set
in which $58\%$ of the edges belong to the original graph, which is best possible.
Also our main result, Theorem \ref{theorem0}, can be rephrased in such a way.

A third motivation for our results is their relation to {\it graph discrepancy} notions
originating in work of Erd\H{o}s et al.~\cite{erfuloso}
and recently considered in \cite{bacsjipl,bacspltr,frhylatr}.
In these works, 
the authors mainly focus 
on the minimum degree threshold ensuring high discrepancy;
considering a much simpler setting, we obtain better estimates,
and illustrate the relation to our results below in Corollary \ref{corollary1}.

\medskip

\noindent If $d\in [0,1]$ is such that $m^+(K)=d{n\choose 2}$, that is, $d$ is the density of the plus-edges in $(K,c)$,
then, since, by symmetry, every edge of $K$ belongs to the same number of subgraphs $G_\pi$, we obtain
\begin{eqnarray}\label{e1}
\frac{1}{n!}\sum_{\pi\in S_n}m^+(G_\pi)=dm(G).
\end{eqnarray}
Furthermore, transposition arguments as in \cite{cayu,mopara} imply that if 
$\sigma_{(K,c)}(G)=\{ m^+_1,\ldots,m^+_k\}$
for $m_1^+<\ldots<m_k^+$,
then 
\begin{eqnarray}\label{e2}
m_{i+1}^+-m_i^+\leq \Delta(G)+\delta(G)\leq 2\Delta(G)\mbox{ for every $i\in [k-1]$,}
\end{eqnarray}
that is, if the maximum degree $\Delta(G)$ of $G$ is small, then there are no big gaps in $\sigma_{(K,c)}(G)$.
This motivates to consider $\max\sigma_{(K,c)}(G)$ and $\min\sigma_{(K,c)}(G)$.

In the case that $c$ is balanced, that is, $d=\frac{1}{2}$,
the existence of copies $G_\pi$ for which $|c(G_\pi)|$ is small, or equivalently,
$m^+(G_\pi)$ is close to $\frac{m(G)}{2}$ has been studied under the term {\it zero sum problems}
or {\it zero sum Ramsey theory} \cite{ca,cahalaza,cayu,ehmora,gage,kisi,mopara,para}.
The observations (\ref{e1}) and (\ref{e2}) are based on common arguments from this area,
and together they imply the existence of some permutation $\pi$ from $S_n$ with 
$$\left|m^+(G_\pi)-d m(G)\right|\leq \Delta(G),$$
that is, the averaging arguments (\ref{e1}) and transformation arguments (\ref{e2})
imply the existence of some permutation $\pi$ from $S_n$ for which $m^+(G_\pi)$ 
is close to its expected value $d m(G)$, when choosing $\pi$ uniformly at random from $S_n$.

Our first result in this paper is that, for bounded maximum degree $\Delta(G)$, and $d\in (0,1)$,
the value $\max\sigma_{(K,c)}(G)$, and, by symmetry, also $\min\sigma_{(K,c)}(G)$,
visibly deviates from $d m(G)$.
Combined with (\ref{e2}), this implies that $\sigma_{(K,c)}(G)$ stretches in bounded discrete steps 
over a non-trivial interval depending on $d$ and $\Delta(G)$.

\begin{theorem}\label{theorem0}
If $K$ is a complete graph of order $n$ with $n\geq 4$, 
$c:E(K)\to\{ \pm 1\}$ is a $\pm 1$-edge labeling of $K$ such that $m^+(K)=d{n\choose 2}$,
and $G$ is a spanning subgraph of $K$ of maximum degree at most $\Delta$,
then there is a permutation $\pi$ from $S_n$ with 
$$m^+(G_\pi)\geq 
\begin{cases}
\left(d+\frac{2-d-2\sqrt{1-d}}{2\Delta+1}-\frac{3}{n-3}\right)m(G), \mbox{ if $d{n\choose 2}\leq \frac{8n^2-14n+3}{25}$, and}\\[3mm]
\left(d+\frac{\sqrt{d}-d}{2\Delta+1}-\frac{3}{n-3}\right)m(G), \mbox{ otherwise}.
\end{cases}$$
\end{theorem}
All proofs are given in the second section.

We now focus in more detail on the balanced case, that is, 
in $(K,c)$ there are equally many plus-edges and minus-edges, or, equivalently, $d=\frac{1}{2}$.
Note that $K$ necessarily has an even number of edges in this case,
which implies that $n$ is equivalent to $0$ or $1$ modulo $4$.
Let ${\cal G}(n,\Delta)$ be the set of all triples $(K,c,G)$ such that 
$K$ is a complete graph of order $n$,
$c$ is a balanced $\pm 1$-edge labeling of $K$,
and $G$ is a spanning subgraph of $K$ of maximum degree at most $\Delta$.

If
$$c_\Delta=
\liminf\limits_{n\to\infty}
\left(\min\left\{
\max\left\{ \frac{m^+(G_\pi)}{m(G)}:\pi\in S_n\right\}:(K,c,G)\in {\cal G}(n,\Delta)\right\}\right),
$$
then Theorem \ref{theorem0} implies
$$c_\Delta\geq \frac{1}{2}+\frac{3-2\sqrt{2}}{4\Delta+2}.$$
Choosing $n$ as a multiple of $4$ and $\Delta+1$,
choosing $G$ as the disjoint union of copies of $K_{\Delta+1}$, 
and choosing $c$ such that the plus-edges of $K$ form the graph that arises 
by removing a matching of size $\frac{n}{4}$ from the complete bipartite graph $K_{\frac{n}{2},\frac{n}{2}}$
implies that 
$$c_\Delta\leq \frac{1}{2}+\frac{1}{2\Delta},$$
that is, there is little room for improvements of Theorem \ref{theorem0}.

A classical result of Erd\H{o}s and Gallai, Theorem 4.1 in \cite{erga}, 
states that a graph $G$ of order $n$, size $m$, 
and matching number $\nu$ satisfies 
\begin{eqnarray}\label{e3}
\nu & \geq & 
\begin{cases}
n-\frac{1}{2}-\sqrt{n^2-2m-n+\frac{1}{4}} & \mbox{, if $m\leq \frac{8n^2-14n+3}{25}$}\\[3mm]
\frac{1}{4}\left(\sqrt{8m+1}-1\right) & \mbox{, otherwise}.
\end{cases}
\end{eqnarray}
Erd\H{o}s and Gallai also show that (\ref{e3}) is best possible, which easily implies that 
$$c_1=2-\sqrt{2}\approx 0.58.$$
Our second result concerns $c_2$.

\begin{theorem}\label{theorem1}
Let $K$ be a complete graph of order $n$
and let $c:E(K)\to\{ \pm 1\}$ be a balanced $\pm 1$-edge labeling of $K$.
\begin{enumerate}[(i)]
\item If $n\geq 4$, 
then there is a Hamiltonian cycle $C$ of $K$ with 
$$m^+(C)\geq \left(2-\sqrt{2}\right)n-o(n)\approx 0.58n-o(n).$$
\item If $n\equiv 0\mod 3$, 
then there is a $C_3$-factor $F$ of $K$ with 
$$m^+(F)\geq \left(\frac{3\sqrt{2}}{4}-\frac{1}{2}\right)n-o(n)\approx 0.56n-o(n).$$
\end{enumerate}
\end{theorem}
In the setting of Theorem \ref{theorem1}, 
the inequality (\ref{e3}) easily implies the existence of some $2$-factor $H$ of $K$,
with $m^+(H)\geq \left(2-\sqrt{2}\right)n-o(n)$,
obtained by extending the union of two disjoint matchings of size $\left(1-\frac{1}{\sqrt{2}}\right)n-o(n)$
in the spanning subgraph of $K$ formed by the plus-edges.
Nevertheless, this argument does not allow any control of the structure of $H$.
It is conceivable that $c_2$ equals $c_1$, that is, $c_2=2-\sqrt{2}\approx 0.58$.
Choosing $G$ as the disjoint union of copies of $K_4$, 
and choosing $c$ such that the minus-edges essentially form a clique of order $\frac{n}{\sqrt{2}}$,
which corresponds to one of the extremal configurations for the estimate (\ref{e3}) of Erd\H{o}s and Gallai,
it follows that $c_3\leq 1-\frac{\sqrt{2}}{3}\approx 0.53$.

Before we proceed to the proofs,
we illustrate the relation of our results 
to the discrepancy notions mentioned above.
Following Balogh et al.~\cite{bacsjipl},
the term (\ref{edisc}) below can be considered the 
{\it discrepancy of Hamiltonian cycles 
in the complete graph $K$}.
In the setting considered in Corollary \ref{corollary1},
their Theorem 1 from \cite{bacsjipl} implies that (\ref{edisc}) 
is at least $\frac{1}{128}n-o(n)\approx 0.0078n-o(n)$.

\begin{corollary}\label{corollary1}
If $K$ is a complete graph of order $n$ with $n\geq 4$, then
\begin{eqnarray}\label{edisc}
\min\Big\{\max\Big\{|c(C)|:C\mbox{ is a Hamiltonian cycle in }K\Big\}:c\mbox{ is a $\pm 1$-edge labeling of $K$}\Big\}
\end{eqnarray}
is at least $\left(3-2\sqrt{2}\right)n-o(n)\approx 0.17n-o(n)$.
\end{corollary}
\begin{proof}
Let the $\pm 1$-edge labeling $c_0$ of $K$ minimize the maximum value of $|c_0(C)|$,
where $C$ is a Hamiltonian cycle in $K$,
that is, (\ref{edisc}) equals 
$\max\{|c_0(C)|:C\mbox{ is a Hamiltonian cycle in }K\}$.
Clearly, we may assume that the number of plus-edges under $c_0$ is at least the number of minus-edges, that is, $|c_0^{-1}(1)|\geq |c_0^{-1}(-1)|$.

For simplicity, we first assume that $n\,{\rm mod}\,4\in \{ 0,1\}$, 
which implies that $K$ has an even number of edges.
Let the $\pm 1$-edge labeling $c_1$ of $K$ arise from $c_0$
by changing $\frac{1}{2}(|c_0^{-1}(1)|-|c_0^{-1}(-1)|)$
of the $+1$-labels on edges to $-1$-labels,
which implies that $c_1$ is balanced.
By Theorem \ref{theorem1}(i),
there is a Hamiltonian cycle $C$ in $K$ 
with $m_{c_1}^+(C)\geq \left(2-\sqrt{2}\right)n-o(n)$,
where the index indicates with respect to which labeling we count the plus-edges.
By construction,
\begin{eqnarray*}
c_0(C) 
& \geq & m_{c_0}^+(C)-m_{c_0}^-(C)
\geq m_{c_1}^+(C)-m_{c_1}^-(C)
= 2m_{c_1}^+(C)-n
\geq \left(3-2\sqrt{2}\right)n-o(n),
\end{eqnarray*}
and, hence, (\ref{edisc}) is at least $\left(3-2\sqrt{2}\right)n-o(n)$.

If $n\,{\rm mod}\,4\not\in \{ 0,1\}$, then removing one or two vertices
yields a complete graph $K'$ of order $n'$ with $n'\,{\rm mod}\,4\in \{ 0,1\}$.
As above, we obtain the existence of a Hamiltonian cycle $C'$ in $K'$ with $c_0(C')\geq \left(3-2\sqrt{2}\right)n'-o(n')$.
Replacing one edge of $C'$ with two or three edges including the removed vertices into $C'$ yields a Hamiltonian cycle $C$ in $K$ with 
$c_0(C)\geq \left(3-2\sqrt{2}\right)n'-o(n')-4=\left(3-2\sqrt{2}\right)n-o(n)$.
Hence, also in this case (\ref{edisc}) is at least $\left(3-2\sqrt{2}\right)n-o(n)$.
\end{proof}
Similarly, Theorem \ref{theorem0} implies that the suitably defined
{\it discrepancy of a fixed spanning subgraph 
with $m$ edges and maximum degree at most $\Delta$
in a complete graph of order $n$} is at least
$\left(\frac{3-2\sqrt{2}}{2\Delta+1}-O\left(\frac{1}{n}\right)\right)m$.

\section{Proofs}

\begin{proof}[Proof of Theorem \ref{theorem0}]
First, we assume that $n$ is even.

Let $d^*$ be such that $d^*{n\choose 2}=\frac{8n^2-14n+3}{25}$.
Since $n\geq 4$, it follows that $\frac{1}{2}\leq d^*\leq \frac{16}{25}$.

Note that 
$$
n-\frac{1}{2}-\sqrt{n^2-2d{n\choose 2}-n+\frac{1}{4}}
=n-\frac{1}{2}-\sqrt{(1-d)n^2-\underbrace{\left((1-d)n-\frac{1}{4}\right)}_{\geq 0}}
\geq n-\frac{1}{2}-\sqrt{(1-d)n^2}
$$
for $d\leq d^*$ and $n\geq 4$.
Furthermore, note that the expression
$$\left(\frac{1}{4}\left(\sqrt{8d{n \choose 2}+1}-1\right)\right)-\left(\frac{\sqrt{d}}{2}n-\frac{1}{2}\right)$$
is decreasing with respect to $d$ for $d$ in $\left[d^*,1\right]$,
and that it equals $0$ for $d=1$.
Therefore, by (\ref{e3}), there is a perfect matching $M_K$ in $K$ such that 
\begin{eqnarray*}
|M_K\cap c^{-1}(1)| & \geq & 
\begin{cases}
\left(1-\sqrt{1-d}\right)n-\frac{1}{2}& \mbox{, if $d\leq d^*$, and}\\[3mm]
\frac{\sqrt{d}}{2}n-\frac{1}{2} & \mbox{, otherwise}.
\end{cases}
\end{eqnarray*}
For $p=\frac{|M_K\cap c^{-1}(1)|}{|M_K|}$, this implies
\begin{eqnarray}\label{epp}
p & \geq & 
\begin{cases}
2-2\sqrt{1-d}-\frac{1}{n},& \mbox{ if $d\leq d^*$, and}\\[3mm]
\sqrt{d}-\frac{1}{n}, & \mbox{ otherwise}.
\end{cases}
\end{eqnarray}
Since $G$ has maximum degree at most $\Delta$,
a simple greedy argument implies that a maximum matching $M^0_G$ in $G$ satisfies
\begin{eqnarray}\label{em0}
|M_G^0| \geq \frac{m(G)}{2\Delta-1}.
\end{eqnarray}
Let $M_G$ be a perfect matching in $K=G\cup \overline{G}$ containing $M_G^0$,
that is, we extend $M_G^0$ by addding edges from $\overline{G}$.

Now, instead of choosing $\pi$ from $S_n$ uniformly at random, 
which leads to $\mathbb{E}[m^+(G_\pi)]\stackrel{(\ref{e1})}{=}dm(G)$,
we change the random choice of $\pi$ as follows in order to exploit $M_K$ and $M_G$:
\begin{itemize}
\item We bijectively assign the $\frac{n}{2}$ edges in $M_G$ 
uniformly at random to the $\frac{n}{2}$ edges in $M_K$,
that is, each of the $\left(\frac{n}{2}\right)!$ assignments is equally likely.
\item If an edge $uv$ from $M_G$ is assigned to an edge $xy$ from $M_K$,
then we choose $\pi$ from $S_n$ such that 
$(\pi(u),\pi(v))=(x,y)$ or 
$(\pi(u),\pi(v))=(y,x)$ equally likely.
Considering all $\frac{n}{2}$ edges of the perfect matchings, this leads to $2^{\frac{n}{2}}$ many possibilities.
\end{itemize}
Altogether, 
we choose the permutation $\pi$ uniformly at random from a subset of $\left(\frac{n}{2}\right)!2^{\frac{n}{2}}$
permutations from $S_n$.

For $uv\in M_G^0$, we obtain
\begin{eqnarray*}
\mathbb{P}[c(\pi(u)\pi(v))=1] & = & p,
\end{eqnarray*}
because the fraction of plus-edges in $M_K$ is exactly $p$.

For $uv\in E(G)\setminus M_G^0$, 
note that there are exactly $\left(d{n\choose 2}-\frac{pn}{2}\right)$ plus-edges in $K-M_K$,
and that $\pi(u)\pi(v)$ equals each of these with probability 
$\frac{2\left(\frac{n}{2}-2\right)!2^{\frac{n}{2}-2}}{\left(\frac{n}{2}\right)!2^{\frac{n}{2}}}$:
In fact, for $\pi(u)\pi(v)$ to equal some plus-edge $xy$ in $K-M_K$, 
the two edges from $M_G$ containing $u$ and $v$ have to be assigned to the two edges from $M_K$ containing $x$ and $y$,
and $\{ \pi(u),\pi(v)\}$ has to equal $\{ x,y\}$. 
There are exactly two possibilities for this.
The remaining $\frac{n}{2}-2$ edges from $M_G$ can be mapped onto the remaining $\frac{n}{2}-2$ edges from $M_K$
without any further restriction.
There are exactly $\left(\frac{n}{2}-2\right)!2^{\frac{n}{2}-2}$ possibilities for this.

We obtain that 
\begin{eqnarray*}
\mathbb{P}[c(\pi(u)\pi(v))=1] & = & \frac{2\left(\frac{n}{2}-2\right)!2^{\frac{n}{2}-2}}{\left(\frac{n}{2}\right)!2^{\frac{n}{2}}}\left(d{n\choose 2}-\frac{pn}{2}\right)
\stackrel{p\leq 1}{\geq} d-\frac{1-d}{n-2}
\stackrel{d\geq 0}{\geq} d-\frac{1}{n-2}
\end{eqnarray*}
Note that
\begin{eqnarray}\label{ep2}
d &\leq &
\begin{cases}
2-2\sqrt{1-d} & \mbox{, if $d\leq d^*$, and}\\[3mm]
\sqrt{d} & \mbox{, otherwise}.
\end{cases}
\end{eqnarray}
If $d\leq d^*$, then linearity of expectation implies
\begin{eqnarray*}
\mathbb{E}[m^+(G_\pi)] & \geq & p|M^0_G|
+\left(d-\frac{1}{n-2}\right)\left(m(G)-|M^0_G|\right)\\
& \stackrel{(\ref{epp})}{\geq} & \left(2-2\sqrt{1-d}-\frac{1}{n}\right)|M^0_G|+\left(d-\frac{1}{n-2}\right)\left(m(G)-|M^0_G|\right)\\
&\geq &\left(2-2\sqrt{1-d}\right)|M^0_G|+d\left(m(G)-|M^0_G|\right)-\frac{m(G)}{n-2}\\
&\stackrel{(\ref{em0}),(\ref{ep2})}{\geq} &\left(2-2\sqrt{1-d}\right)\frac{m(G)}{2\Delta+1}+d\left(m(G)-\frac{m(G)}{2\Delta+1}\right)-\frac{m(G)}{n-2}\\
&=& \left(d+\frac{2-d-2\sqrt{1-d}}{2\Delta+1}-\frac{1}{n-2}\right)m(G).
\end{eqnarray*}
Similarly, if $d>d^*$, then linearity of expectation and a similar estimation as above imply
\begin{eqnarray*}
\mathbb{E}[m^+(G_\pi)] 
&\stackrel{(\ref{epp}),(\ref{em0}),(\ref{ep2})}{\geq} &\sqrt{d}\frac{m(G)}{2\Delta+1}+d\left(m(G)-\frac{m(G)}{2\Delta+1}\right)-\frac{m(G)}{n-2}\\
&=& \left(d+\frac{\sqrt{d}-d}{2\Delta+1}-\frac{1}{n-2}\right)m(G),
\end{eqnarray*}
which completes the proof in the case that $n$ is even.

Now, let $n$ be odd.
There is a vertex $x$ of $K$ such that $m^+(K-x)\geq d{n-1\choose 2}$.
Possibly replacing $G$ by an isomorphic copy, 
we may assume that $x$ is a vertex of minimum degree in $G$,
that is, $m(G-x)=m(G)-\delta(G)\geq \left(1-\frac{2}{n}\right)m(G)$.
Therefore, applying the above estimates to $\left(K-x,c\mid_{E(K-x)}\right)$ and $G-x$, 
we obtain 
\begin{eqnarray*}
\mathbb{E}[m^+(G_\pi)] 
&\geq&
\begin{cases}
\left(d+\frac{2-d-2\sqrt{1-d}}{2\Delta+1}-\frac{1}{n-3}\right)(m(G)-\delta(G)), \mbox{ if $d\leq d^*$, and}\\[3mm]
\left(d+\frac{\sqrt{d}-d}{2\Delta+1}-\frac{1}{n-3}\right)(m(G)-\delta(G)), \mbox{ otherwise}
\end{cases}\\[3mm]
&\geq&
\begin{cases}
\left(d+\frac{2-d-2\sqrt{1-d}}{2\Delta+1}-\frac{1}{n-3}\right)\left(1-\frac{2}{n}\right)m(G), \mbox{ if $d\leq d^*$, and}\\[3mm]
\left(d+\frac{\sqrt{d}-d}{2\Delta+1}-\frac{1}{n-3}\right)\left(1-\frac{2}{n}\right)m(G), \mbox{ otherwise}
\end{cases}\\[3mm]
&\geq&
\begin{cases}
\left(d+\frac{2-d-2\sqrt{1-d}}{2\Delta+1}-\frac{3}{n-3}\right)m(G), \mbox{ if $d\leq d^*$, and}\\[3mm]
\left(d+\frac{\sqrt{d}-d}{2\Delta+1}-\frac{3}{n-3}\right)m(G), \mbox{ otherwise},
\end{cases}
\end{eqnarray*}
which completes the proof.
\end{proof}

\begin{proof}[Proof of Theorem \ref{theorem1}(i)]
Let $G$ be the spanning subgraph of $K$ formed by the plus-edges,
that is, the graph $G$ equals $(V(K),c^{-1}(1))$.
In view of the desired statement, we may assume that $n\geq 10$.
It is easy to see that the desired statement is equivalent (up to the specific choice of the $o(n)$ term) to 
the existence of non-trivial disjoint paths $P_1,\ldots,P_k$ in $G$
such that 
$$m(P_1)+\cdots+m(P_k)\geq 2n + 3 - \sqrt{2n^2 + 14n + 1}.$$
In fact, removing the minus-edges from a Hamiltonian cycle $C$ as in the statement yields such paths
(as well as some isolated vertices),
and, conversely, 
such paths (and the remaining isolated vertices)
can easily be concatenated with edges from $K$ to form a Hamiltonian cycle of $K$.
Therefore, let the non-trivial disjoint paths $P_1,\ldots,P_k$ in $G$ be chosen such that
\begin{itemize}
\item $m(H)$ is as large as possible, where $H=P_1\cup\ldots\cup P_k$, and
\item subject to the first condition, the number $k$ of paths is as small as possible.
\end{itemize}
For a contradiction, we suppose that $m(H)<2n + 3 - \sqrt{2n^2 + 14n + 1}$.
For $d\in \{ 1,2\}$, let $V_d$ be the set of vertices that have degree $d$ in $H$,
and, let $n_d=|V_d|$.
Let $V_0=V(G)\setminus (V_1\cup V_2)$, and $n_0=|V_0|$.
Let $P_i$ have the endvertices $x_i$ and $y_i$ for every $i\in [k]$,
that is, $V_1=\{ x_1,y_1,\ldots,x_k,y_k\}$.
Note that $n_1=2k$, $n_2=m(H)-k$, and $n_0=n-n_1-n_2=n-m(H)-k$.

\begin{claim}\label{claim1}
$m(G[V_0\cup V_1])\leq k$.
\end{claim}
\begin{proof}
If $uv$ is an edge of $G[V_0\cup V_1]$ that does not belong to $\{ x_iy_i:i\in [k]\}$,
then $H+uv$ is the union of non-trivial disjoint paths in $G$ with more edges than $H$,
which is a contradiction. 
Therefore, the edge set of $G[V_0\cup V_1]$ is contained in $\{ x_iy_i:i\in [k]\}$.
\end{proof}
The components of $H-V_1$ are paths $Q_1,\ldots,Q_\ell$ with $\ell\leq k$,
where we may assume that $Q_i=P_i-x_i-y_i$.
Note that a $Q_i$ may be trivial, that is, consist of just one vertex only.
If $\ell=0$, then $V(K)=V_0\cup V_1$, and Claim \ref{claim1} implies
$m(G)=k\leq \frac{n}{2}<\frac{1}{2}{n\choose 2}$,
which contradicts the hypothesis that $c$ is balanced.
Since $G$ contains $\ell$ disjoint paths $P_1,\ldots,P_\ell$ of order at least $3$, we have $\ell\leq \frac{n}{3}$.

\begin{claim}\label{claim2}
For every vertex $u$ from $V_0$, 
there are at most $\frac{1}{2}(n_2-\ell+2)$ edges in $G$ between $u$ and $V_2$,
and $u$ is adjacent to an endvertex in at most one of the paths $Q_1,\ldots,Q_\ell$.
\end{claim}
\begin{proof}
If $u$ is adjacent to two consecutive vertices $v$ and $w$ of some
$Q_i$, then $H-vw+vu+uw$ is the union of non-trivial disjoint paths in
$G$ with more edges than $H$, which is a contradiction.  Hence, the
vertex $u$ has at most $\frac{n(Q_i)+1}{2}$ many neighbors in $V(Q_i)$
for every $i$ in $[\ell]$.  If there are two distinct paths $Q_i$ and
$Q_j$ such that $u$ is adjacent to an endvertex $x_i'$ in $Q_i$ as
well as an endvertex $x_j'$ in $Q_j$, then, by symmetry, we may assume
that $x_i$ is a neighbor of $x_i'$ in $H$, and $x_j$ is a neighbor of
$x_j'$ in $H$, and $H-x_ix_i'-x_jx_j'+x_i'u+x_j'u$ is the union of
less than $k$ non-trivial disjoint paths in $G$ with the same number
of edges as $H$, which is a contradiction.  
So, assuming that $u$ is adjacent to an endvertex of $Q_i$, 
for every $i'\in [\ell]\setminus \{i\}$, 
the vertex $u$ is adjacent to at most $\frac{n(Q'_i)-1}{2}$ vertices of $Q'_{i'}$.
This implies that the number of neighbors of $u$ in $V_2$ is at most 
$$\frac{n(Q_i)+1}{2}+\sum\limits_{j\in [\ell]\setminus \{ i\}}\frac{n(Q_j)-1}{2}=\frac{n_2-\ell+2}{2}$$
Clearly, if $u$ is not adjacent to any endvertex of a path $Q_i$, this upper bound also holds.
\end{proof}
Let $d$ be the average number of neighbors in $V_2$ of the vertices in $V_0$,
that is, there are $dn_0$ edges in $G$ between $V_0$ and $V_2$ altogether.
Claim \ref{claim2} implies 
\begin{eqnarray}\label{ed1}
d\leq \frac{n_2-\ell+2}{2}\leq \frac{n_2+1}{2}.
\end{eqnarray}

\begin{claim}\label{claim3}
There are at least ${d-1\choose 2}$ non-edges in $G[V_2]$.
\end{claim}
\begin{proof}
By the definition of $d$, there is a vertex $u$ from $V_0$ that has at least $d$ neighbors in $V_2$.
By Claim \ref{claim2}, the vertex $u$ is adjacent to an endvertex in at most one of the paths $Q_1,\ldots,Q_\ell$.
Therefore, removing at most one of the neighbors of $u$ in $V_2$, which is an endvertex of some $Q_i$,
yields a set $N$ of at least $d-1$ neighbors of $u$ in $V_2$ 
such that there is an orientation of the $Q_i$
for which the set of in-neighbors $N^-$ of the vertices from $N$ all belong to the paths $Q_i$.
If $v^-w^-$ is an edge in $G$ for two vertices $v^-$ and $w^-$ from $N^-$
that are the in-neighbors of $v$ and $w$ from $N$, respectively,
then $H-vv^--ww^-+v^-w^-+vu+uw$
is the union of non-trivial disjoint paths in $G$ with more edges than $H$,
which is a contradiction. 
Hence, the set $N^-$ is independent, which implies the claim.
\end{proof}

\begin{claim}\label{claim4}
If $k\geq 2$, then $G$ contains at most $\frac{1}{2}n_1(n_2+\ell)$ edges between $V_1$ and $V_2$.
\end{claim}
\begin{proof}
Let $i,j\in [k]$.
Let $P_j:x_ju_1\ldots u_py_j$, that is, $Q_j$ is $u_1\ldots u_p$.
If there is some $q\in [p-1]$ such that $x_i$ is adjacent to $u_{q+1}$ and $y_{i+1}$ is adjacent to $u_q$,
where we identify $y_{k+1}$ with $y_1$,
then $H-u_qu_{q+1}+x_iu_{q+1}+y_{i+1}u_q$ is the union of non-trivial disjoint paths in $G$ 
with more edges than $H$, which is a contradiction. 
Note that this is also true even if $j=i$ or $j=i+1$.
Therefore, there are at most $p+1$ edges in $G$ between $\{ x_i,y_{i+1}\}$ and $V(Q_j)$,
which easily implies the statement.
\end{proof}
We are now in a position to estimate the total number $m(G)$ of edges of $G$
in order to derive a contradiction.

First, we assume that $k=1$.
In this case, the claims imply
\begin{eqnarray*}
m(G) & \leq & 1+dn_0+{n_2\choose 2}-{d-1\choose 2}+n_1n_2.
\end{eqnarray*}
Considered as a quadratic function of $d$, 
the right hand side is maximized for $d=n+\frac{1}{2}-m(H)>\frac{m(H)}{2}$.
Hence, the function is increasing for $d\stackrel{(\ref{ed1})}{\leq} \frac{n_2+1}{2}=\frac{m(H)}{2}$,
and substituting $d=\frac{m(H)}{2}$ yields
\begin{eqnarray*}
m(G) & \leq & -\frac{1}{8}m(H)^2+\left(\frac{4n+6}{8}\right)m(H)-1,
\end{eqnarray*}
which, for $m(H)<2n + 3 - \sqrt{2n^2 + 14n + 1}$,
is strictly less than $\frac{1}{2}{n\choose 2}$,
a contradiction.

Next, let $k\geq 2$.
In this case, the claims imply
\begin{eqnarray*}
m(G)&\leq & k+dn_0+{n_2\choose 2}-{d-1\choose 2}+\frac{1}{2}n_1\left(n_2+\frac{n}{3}\right),
\end{eqnarray*}
where the last term $n_1n_2$ has been improved using Claim \ref{claim4} and $\ell\leq \frac{n}{3}$.

We consider two cases according to the value of $k$.

First, we suppose that $k\leq 2n-3m(H)+2$.
In this case, considered as a quadratic function of $d$, 
the upper bound on $m(G)$ is maximized for $d=\frac{n_2+1}{2}=\frac{m(H)-k+1}{2}$.
Substituting this value for $d$, we obtain
\begin{eqnarray*}
m(G)&\leq & -\frac{1}{8}k^2-\left(\frac{2n-3m(H)-6}{12}\right)k+\frac{(4n-m(H)-3)(m(H)+1)}{8}.
\end{eqnarray*}
Now, considered as a quadratic function of $k$,
this upper bound is maximized for $k=2$.
Substituting this value for $k$, we obtain
\begin{eqnarray*}
m(G)&\leq & \frac{1}{6}n+\frac{1}{8}-\frac{1}{8}m(H)^2+\frac{1}{2}m(H)n,
\end{eqnarray*}
which, for $m(H)<2n + 3 - \sqrt{2n^2 + 14n + 1}$ and $n\geq 10$,
is strictly less than $\frac{1}{2}{n\choose 2}$,
a contradiction.

Next, we suppose that $k\geq 2n-3m(H)+3$.
In this case, considered as a quadratic function of $d$, 
the upper bound on $m(G)$ is maximized for $d=n-m(H)-k+\frac{3}{2}$.
Substituting this value for $d$, we obtain
\begin{eqnarray*}
m(G)&\leq & \frac{3}{2}n+ km(H) + m(H)^2 - nm(H) - \frac{2}{3}nk + \frac{1}{2}n^2 + \frac{1}{8} - 2m(H).
\end{eqnarray*}
Now, considered as a linear function of $k$,
this upper bound is decreasing in $k$ for $m(H)<2n + 3 - \sqrt{2n^2 + 14n + 1}$.
Hence, substituting $k=2n-3m(H)+3$ yields
\begin{eqnarray*}
m(G)&\leq & -\frac{5}{6}n^2+\frac{1}{2}(6m(H)-1)n-2m(H)^2+m(H)+\frac{1}{8},
\end{eqnarray*}
which, for $m(H)<2n + 3 - \sqrt{2n^2 + 14n + 1}$,
is strictly less than $\frac{1}{2}{n\choose 2}$,
a contradiction.

This completes the proof.
\end{proof}

\begin{proof}[Proof of Theorem \ref{theorem1}(ii)]
For $(K,c)$ as in the statement, 
let $C_1,\ldots,C_k$ with $3k=n$ be the components of a $C_3$-factor $F$ of $K$.
For $j\in \{ 0,1,2,3\}$, let $t_j$ be such that $t_jn$ is the number of $i$ in $[k]$ with $m^+(C_i)=j$.
In particular, 
\begin{eqnarray}\label{et1}
t_0+t_1+t_2+t_3=\frac{1}{3}.
\end{eqnarray}
We assume that $F$ is chosen in such a way that 
\begin{itemize}
\item $m^+(F)=m^+(C_1)+\cdots+m^+(C_k)$ is as large as possible, and
\item subject to the first condition, the value of $t_2$ is as large as possible.
\end{itemize}
The choice of $F$ allows to upper bound the number of plus-edges between any two of the triangles in $F$.
If $C_i$ and $C_j$ satisfy $m^+(C_i)=m^+(C_j)=0$, 
then any plus-edge between $V(C_i)$ and $V(C_j)$ allows to replace 
$C_i$ and $C_j$ within $F$ by two triangles $C_i'$ and $C_j'$ such that $m^+(C_i)+m^+(C_j)\geq 1$,
which would contradict the choice of $F$.
If $C_i$ and $C_j$ satisfy $m^+(C_i)=0$ and $m^+(C_j)=3$,
and there are at least $4$ plus-edges between these two triangles, 
then there are two such edges, say $e$ and $f$, that are disjoint,
and replacing $C_i$ and $C_j$ within $F$ by two triangles $C_i'$ and $C_j'$ such that 
$C_i'$ contains $e$ and exactly one edge from $C_i$, and
$C_j'$ contains $f$ and exactly one edge from $C_j$,
yields a contradiction either to the first condition or to the second condition 
within the choice of $F$.
If $C_i:xyzx$ and $C_j:x'y'z'x'$ satisfy $m^+(C_i)=2$, $c(xz)=-1$, and $m^+(C_j)=3$,
and there are at least $8$ plus-edges between these two triangles, 
then, by symmetry, we may assume that all minus-edges between these two triangles
are between $\{ x,y\}$ and $\{ x',y'\}$, and
replacing $C_i$ and $C_j$ within $F$ by the two triangles $x'y'zx'$ and $xyz'x$,
contradicts the choice of $F$.
Now, suppose that $C_i:xyzx$ and $C_j:x'y'z'x'$ satisfy $m^+(C_i)=1$, $c(xz)=1$, $m^+(C_j)=2$, and $c(x'z')=-1$.
Considering the two triangles $xy'zx$ and $x'yz'x'$ implies 
that at most two of the four edges $xy'$, $y'z$, $x'y$, and $yz'$ are plus-edges.
Hence, if there are at least $4$ plus-edges between the two triangles, 
then, by symmetry, we may assume that $xx'$ is a plus-edge.
Considering the two triangles $xx'zx$ and $yz'y'y$ implies 
that all the three edges $x'z$, $yy'$, and $yz'$ are minus-edges.
Considering the two triangles $xx'y'x$ and $yzz'y$ implies 
that one of the two edges $xy'$ and $zz'$ is a minus-edges.
Finally,
considering the two triangles $xx'yx$ and $y'z'zy'$ implies 
that one of the two edges $x'y$ and $y'z$ is a minus-edges.
Altogether, these observations yield at least $3+1+1$ minus-edges between $V(C_i)$ and $V(C_j)$,
which implies that there are at most $4$ plus-edges between these two triangles.

The following table summarizes the upper bounds on the number of plus-edges between the different types of triangles in $F$.
Since verifying the correctness of these values is straightforward, we leave the remaining details to the reader.

\begin{center}
\begin{tabular}{|c||c|c|c|c|}\hline
$(i,j)$ & $m^+(C_j)=0$ & $m^+(C_j)=1$ & $m^+(C_j)=2$ & $m^+(C_j)=3$\\ \hline \hline 
$m^+(C_i)=0$ & 0 & 0 & 3 & 3 \\ \hline
$m^+(C_i)=1$ & 0 & 1 & 4 & 6 \\ \hline
$m^+(C_i)=2$ & 3 & 4 & 5 & 7 \\ \hline
$m^+(C_i)=3$ & 3 & 6 & 7 & 9 \\ \hline
\end{tabular}
\end{center}
Since $c$ is balanced, we obtain
\begin{eqnarray*}
\left(\frac{1}{4}-o(1)\right)n^2 & = & \frac{1}{2}{n\choose 2}
=m^+(K)\\
& \leq &
t_1n+2t_2n+3t_3n
+{t_1n\choose 2}
+5{t_2n\choose 2}
+9{t_3n\choose 2}\\
&& 
+3t_0n(t_2n+t_3n)
+4t_1nt_2n
+6t_1nt_3n
+7t_2nt_3n\\
& = &
\left(\frac{1}{2}t_1^2+\frac{5}{2}t_2^2+\frac{9}{2}t_3^2+3t_0(t_2+t_3)+4t_1t_2+6t_1t_3+7t_2t_3+o(1)\right)n^2,
\end{eqnarray*}
and, hence,
\begin{eqnarray}\label{et2}
h(t_0,t_1,t_2,t_3) &:= & \frac{1}{2}t_1^2+\frac{5}{2}t_2^2+\frac{9}{2}t_3^2+3t_0(t_2+t_3)+4t_1t_2+6t_1t_3+7t_2t_3 
\geq \frac{1}{4}-o(1).
\end{eqnarray}
Since $m^+(F)=(t_1+2t_3+3t_3)n$, in order to complete the proof, 
it suffices to show that the optimum value of the following optimization problem is at least 
$\frac{3\sqrt{2}}{4}-\frac{1}{2}-o(1)$:
\begin{eqnarray}\label{et3}
\begin{array}{rccl}
\min & t_1+2t_2+3t_3 & & \\[3mm]
s.th. & t_0+t_1+t_2+t_3 & = & \frac{1}{3}\\[3mm]
 & h(t_0,t_1,t_2,t_3) & \geq &\frac{1}{4}-o(1)\\[3mm]
 & t_0,t_1,t_2,t_3 & \geq & 0
\end{array}
\end{eqnarray}
Since 
$\frac{\partial}{\partial t_0}h(t_0,t_1,t_2,t_3)$ 
is strictly less than 
$\frac{\partial}{\partial t_i}h(t_0,t_1,t_2,t_3)$ for every $i$ in $[3]$
and every feasible solution $(t_0,t_1,t_2,t_3)$ of (\ref{et3}),
every optimum solution $(t_0,t_1,t_2,t_3)$ of (\ref{et3}) 
satisfies the inequality (\ref{et2}) with equality;
otherwise, slightly inceasing $t_0$ and decreasing $t_1+t_2+t_3$ by the same amount
would yield a better feasible solution.
Note 
that 
$\frac{\partial}{\partial t_3}h(t_0,t_1,t_2,t_3)-\frac{\partial}{\partial t_0}h(t_0,t_1,t_2,t_3)
=3t_0 + 6t_1 + 4t_2+6t_3\geq 1$
for every feasible solution $(t_0,t_1,t_2,t_3)$ of (\ref{et3}), 
that 
$\frac{\partial}{\partial t_3}h(t_0,t_1,t_2,t_3)-\frac{\partial}{\partial t_1}h(t_0,t_1,t_2,t_3)
=3t_0 + 5t_1 + 3t_2+3t_3\geq 1$
for every feasible solution $(t_0,t_1,t_2,t_3)$ of (\ref{et3}), and 
that 
$\frac{\partial}{\partial t_3}h(t_0,t_1,t_2,t_3)-\frac{\partial}{\partial t_2}h(t_0,t_1,t_2,t_3)
=2t_1 + 2t_2+2t_3\geq \frac{2}{3}$
for every feasible solution $(t_0,t_1,t_2,t_3)$ of (\ref{et3}).
This implies that, 
for every optimum solution $(t_0,t_1,t_2,t_3)$ of (\ref{et3}),
increasing $t_3$ by $o(1)$ and decreasing $t_0+t_1+t_2$ by the same amount
without violating the condition $t_0,t_1,t_2\geq 0$,
yields a feasible solution for the following optimization problem (\ref{et4}), 
whose objective function value is larger by at most $o(1)$.
\begin{eqnarray}\label{et4}
\begin{array}{rccl}
\min & t_1+2t_2+3t_3 & & \\[3mm]
s.th. & t_0+t_1+t_2+t_3 & = & \frac{1}{3}\\[3mm]
 & h(t_0,t_1,t_2,t_3) & = &\frac{1}{4}\\[3mm]
 & t_0,t_1,t_2,t_3 & \geq & 0
\end{array}
\end{eqnarray}
Since the optimum value of (\ref{et4}) is at least the optimum value of (\ref{et3}),
this implies that the two optimal values differ only by $o(1)$.
Hence, in order to complete the proof, 
it suffices to show that the optimum value of (\ref{et4}) is at least $\frac{3\sqrt{2}}{4}-\frac{1}{2}$.

The equation $h(t_0,t_1,t_2,t_3)=\frac{1}{4}$ allows to express $t_0$ in terms of $t_1$, $t_2$, and $t_3$,
and substituting the corresponding expression into $t_0 + t_1 + t_2 + t_3 =\frac{1}{3}$, allows to 
express $t_2$ in terms of $t_1$ and $t_3$.
Substituting these expressions for $t_0$ and $t_2$, we obtain
$$t_1+2t_2+3t_3=3t_1 + 5t_3 + 2 -\sqrt{8t_1^2 + (32t_3 + 8)t_1 + 16t_3^2 + 16t_3 + 2}=:f(t_1,t_3).$$
It follows that the optimum value of (\ref{et4}) is at least the optimum value of the following optimization problem,
where we implicitly relax the conditions ``$t_0\geq 0$'' and ``$t_2\geq 0$'':
\begin{eqnarray}\label{et5}
\min \left\{ f(t_1,t_3):
t_1,t_3 \geq 0\mbox{ and }t_1+t_3 \leq \frac{1}{3}\right\}
\end{eqnarray}
Since there is no point $(t_1,t_3)$ 
in the interior of $\left\{ (x,y):x,y\geq 0\mbox{ and }x+y\leq \frac{1}{3}\right\}$
for which $\frac{\partial}{\partial t_1}f(t_1,t_3)=\frac{\partial}{\partial t_3}f(t_1,t_3)=0$,
the minimum (\ref{et5}) is assumed on the boundary.

Since $f(t_1,0)=\left(3-2\sqrt{2}\right)t_1 + 2-\sqrt{2}$, we obtain
\begin{eqnarray*}
\min \left\{ f(t_1,0):
0\leq t_1\leq \frac{1}{3}\right\}
&=& f(0,0)= 2-\sqrt{2}>\frac{3\sqrt{2}}{4}-\frac{1}{2}.
\end{eqnarray*}
Since 
$\frac{\partial}{\partial t_3}f(0,t_3)=0$ for $t_3\in \left[0,\frac{1}{3}\right]$
only if $t_3=\frac{5\sqrt{2}}{12}-\frac{1}{2}\approx 0.08$, we obtain
\begin{eqnarray*}
\min \left\{ f(0,t_3):
0\leq t_3\leq \frac{1}{3}\right\}
&=&\min\left\{ f(0,0),f\left(0,\frac{5\sqrt{2}}{12}-\frac{1}{2}\right),f\left(0,\frac{1}{3}\right)\right\}
=\frac{3\sqrt{2}}{4}-\frac{1}{2}.
\end{eqnarray*}
Finally, since
$\frac{\partial}{\partial t_1}f\left(t_1,\frac{1}{3}-t_1\right)=0$ for $t_1\in \left[0,\frac{1}{3}\right]$
only if $t_1=\frac{5\sqrt{6}}{18}-\frac{1}{2}$, we obtain
\begin{eqnarray*}
\min \left\{ f\left(t_1,\frac{1}{3}-t_1\right):0\leq t_1\leq \frac{1}{3}\right\}
&=&\min\left\{
f\left(0,\frac{1}{3}\right),
f\left(\frac{5\sqrt{6}}{18}-\frac{1}{2},\frac{5}{6}-\frac{5\sqrt{6}}{18}\right),f\left(\frac{1}{3},0\right)
\right\}\\
&=& \frac{14}{3}-\frac{5\sqrt{6}}{3}>\frac{3\sqrt{2}}{4}-\frac{1}{2}.
\end{eqnarray*}
Altogether, it follows that the optimum value of (\ref{et5}) is $\frac{3\sqrt{2}}{4}-\frac{1}{2}$,
which implies that the optimum value of (\ref{et4}) is also at least this value.
This completes the proof. 
\end{proof}
It seems possible to apply a similar approach to other graphs whose components 
are all isomorphic, such as, for instance, $K_4$-factors or $K_{1,3}$-factors.

\section{Conclusion}

An obvious task motivated by our results is to determine values $c_\Delta$ more precisely;
at least, for small values of $\Delta$. 
Furthermore, it seems straightforward to generalize Theorem \ref{theorem1}
to not necessarily balanced edge labelings $c$.

There seem to be no immediate directed analogues of our results.
If $D$ is the complete digraph with vertex set $[2n]$,
that is, between every two vertices of $D$, there are both possible arcs,
and $c:A(D)\to\{ \pm 1\}$ is such that 
$c((u,v))=1$ for all arcs $(u,v)$ with $u\in [n]$, and 
$c((u,v))=-1$ for all arcs $(u,v)$ with $u\in [2n]\setminus [n]$,
then there are equally many plus- and minus-arcs,
but every directed Hamiltonian cycle has exactly $n$ plus- and $n$ minus-arcs.
Similarly, there are tournaments $T$ with a unique directed Hamiltonian cycle,
which allows to force all plus-arcs or all minus-arcs in the directed Hamiltonian cycle
even though $T$ has equally many plus- and minus-arcs.

The problems studied in this paper clearly relate to classical and new results 
concerning extremal graph theory, Ramsey theory, and Hamiltonicity.
Further directions that could be pursued may be inspired by \cite{bagu,zhlibr}.

\end{document}